\documentclass[12pt,reqno]{article}

\usepackage[usenames]{color}
\usepackage{amssymb}
\usepackage{amsmath}
\usepackage{amsthm}
\usepackage{amsfonts}
\usepackage{amscd}
\usepackage{graphicx}

\usepackage[colorlinks=true,
linkcolor=webgreen,
filecolor=webbrown,
citecolor=webgreen]{hyperref}

\definecolor{webgreen}{rgb}{0,.5,0}
\definecolor{webbrown}{rgb}{.6,0,0}

\usepackage{color}
\usepackage{fullpage}
\usepackage{float}

\usepackage{graphics}

\usepackage{latexsym}
\usepackage{epsf}
\usepackage{breakurl}

\newcommand{\seqnum}[1]{\href{https://oeis.org/#1}{\rm \underline{#1}}}

\begin{document}
	
\theoremstyle{plain}
\newtheorem{theorem}{Theorem}
\newtheorem{corollary}[theorem]{Corollary}
\newtheorem{lemma}[theorem]{Lemma}
\newtheorem{proposition}[theorem]{Proposition}

\theoremstyle{definition}
\newtheorem{definition}[theorem]{Definition}
\newtheorem{example}[theorem]{Example}
\newtheorem{conjecture}[theorem]{Conjecture}

\theoremstyle{remark}
\newtheorem{remark}[theorem]{Remark}

\begin{center}
\vskip 1cm{\LARGE\bf Evaluating $\zeta(s)$ At Odd Positive Integers Using Automatic Dirichlet Series} 
\vskip 1cm \large
L\'aszl\'o T\'oth\\
L-8476 Eischen \\
Grand Duchy of Luxembourg\\
\href{mailto:uk.laszlo.toth@gmail.com}{\tt uk.laszlo.toth@gmail.com}
\end{center}
\vskip .2 in

\begin{abstract}
In this paper, we use the Thue-Morse sequence and the paperfolding sequence to build a Dirichlet series that evaluates to a linear combination of the Riemann zeta function at odd positive integers and odd powers of $\pi$. In doing so, we also provide an alternative proof of a 2015 result by Allouche and Sondow.
\end{abstract}
	
\section{Introduction}

Let 
$$
\zeta(s) = \sum_{n\geq1} \frac{1}{n^s}
$$
denote the Riemann zeta function for positive integers $s>1$. The values of $\zeta$ at even integers $s=2k$ are expressible in terms of even powers of $\pi$ and the Bernoulli numbers $B_s$ thanks to Euler's classical formula (see, for instance, Apostol \cite[Thm.\ 12.17, p.\ 266]{Apostol76}),
$$
\zeta (2k)=\frac {(-1)^{k+1} B_{2k}(2\pi )^{2k}}{2(2k)!}.
$$
On the other hand, no such formula is known for odd integers $s=2k+1$, and several authors, such as Waldschmidt \cite{Waldschmidt10}, conjectured that for $k\geq1$, the numbers $\zeta(2k+1)$ and $\pi$ are algebraically independent. In fact, the arithmetic nature of $\zeta(2k+1)$ is still not fully understood. Ap\'ery \cite{Apery79} showed that $\zeta(3)$ is irrational and later, Rivoal \cite{Rivoal00} and Ball and Rivoal \cite{Ball01} proved that $\zeta(2k+1)$ is irrational for infinitely many $k\geq1$. Zudilin \cite{Zudilin01} then showed that at least one of the numbers $\zeta(5)$, $\zeta(7)$, $\zeta(9)$, $\zeta(11)$ is irrational, and then a year later \cite{Zudilin02} that for any integer $k\geq0$, at least one number in $\zeta(2k + 3), \zeta(2k + 5),\ldots, \zeta(16k + 7)$ is irrational. More recently, Rivoal and Zudilin \cite{Rivoal20} showed that there exist at least two irrational numbers amongst the $33$ odd zeta values $\zeta(5), \zeta(7),\ldots, \zeta(69)$.

Some representations of the zeta function at odd positive integers involve odd powers of $\pi$. One example, given by N\"orlund \cite[Eq.\ 81*, p.\ 66]{Norlund24}, is
$$
\zeta(2k+1)={\frac {(-1)^{k+1} (2\pi )^{2k+1}}{2(2k+1)!}} \int_0^1 B_{2k+1}(x) \cot (\pi x) dx,
$$
where the $B_{2k+1}$ are the Bernoulli polynomials. A recent proof of this equality was given by Cvijovi\'{c} and Klinowski \cite{Cvijovic02}.

Other formulas exist that are specific to certain odd integers. One such example is Ramanujan's \cite{Berndt89} classical formula for $\zeta(3)$,
$$
\zeta (3)={\frac {7}{180}}\pi ^{3}-2\sum _{k=1}^{\infty }{\frac {1}{k^{3}(e^{2\pi k}-1)}}.
$$
Plouffe \cite{Plouffe11} discovered several similar identities involving odd powers of $\pi$, including
$$
\zeta (7)={\frac {19}{57600}}\pi ^{7}-2\sum _{k=1}^{\infty }{\frac {1}{k^{7}(e^{2\pi k}-1)}}.
$$
However, no closed form formula exists today involving $\zeta(2k+1)$ and $\pi^{2k+1}$ akin to Euler's identity for $\zeta(2k)$ and $\pi^{2k}$.

\subsection{Scope of this paper}
In this paper we look at this problem by constructing Dirichlet series involving two binary sequences; first, the classical Thue-Morse sequence $(t_n)_{n\geq0}$, beginning with
$$
0, 1, 1, 0, 1, 0, 0, 1, \ldots ,
$$
available in the On-Line Encyclopedia of Integer Sequences (OEIS) \cite{oeis} as sequence \seqnum{A010060} and admitting the recurrence relation
$$
t_{2n} = t_n, \ \ \ \ t_{2n+1} = 1-t_n.
$$
Second, the paperfolding or ``dragon-curve'' sequence $(b_n)_{n\geq1}$, starting with
$$
0, 0, 1, 0, 0, 1, 1, 0, \ldots
$$
and defined by the recurrence
$$
b_{2n} = b_n, \ \ \ \ b_{4n+1} = 0, \ \ \ \ b_{4n+3} = 1,
$$
listed in the OEIS as \seqnum{A014707} (with a few variations, such as \seqnum{A014577}). A few examples of our results are shown below.
\begin{example}
We have the following equalities.
\begin{align*}
(i) \ \ \ \ &
\zeta(3) - \pi^3 = \sum_{n\geq1} \frac{8^{-1}(9 t_{n-1} + 7 t_n) + 28(2b_n-1)}{n^3}, \\
(ii) \ \ \ \ & \zeta(5) - \frac 53 \pi^5 = \sum_{n\geq1} \frac{32^{-1}(33 t_{n-1} + 31 t_n) + 496(2b_n-1)}{n^5}, \\
(iii) \ \ \ \ & \zeta(7) - \frac{122}{45} \pi^7 = \sum_{n\geq1} \frac{128^{-1}(129 t_{n-1} + 127 t_n) + 8128(2b_n-1)}{n^7}.
\end{align*}
\end{example}
In fact, we prove the following theorem.
\begin{theorem} \label{theo-zeta-pi}
For all positive integers $k\geq1$, we have
$$
\zeta(2k+1) - \frac{2^{2k-1} |E_{2k}| }{(2k)!}  \pi^{2k+1} = \sum_{n\geq1} \frac{N(n;k)}{n^{2k+1}},
$$
where
$$
N(n;k) = (2^{4k+1} - 2^{2k}) (2 b_n -1) + (2^{-({2k+1})}) ((2^{2k+1}+1) t_{n-1} + (2^{2k+1}-1) t_n)
$$
and $E_k$ is the $k^{\rm th}$ Euler number defined by
$$
\frac{1}{\cosh t} = \sum_{k\geq0} \frac{E_{k}}{k!} t^{k}.
$$
\end{theorem}
We note that Dirichlet series whose coefficients are linear combinations of $2$-automatic sequences have been widely studied in the past. A few examples involving the Thue-Morse sequence are
$$
\sum_{n\geq0} \frac{\varepsilon_n}{(n+1)^s} = \sum_{k\geq1} 2^{-s-k} \binom{s+k-1}{k} \sum_{n\geq0} \frac{\varepsilon_n}{(n+1)^{s+k}},
$$
with $\varepsilon_n = (-1)^{t_n}$ for all $n\geq0$ and valid for all $\Re(s)>1$, due to Allouche and Cohen \cite{Allouche85}, which was subsequently continued by Allouche, Mend\`es France, and Peyri\`ere \cite{Allouche00}, and the identity
$$
\sum_{n\geq0} \frac{\varepsilon_n}{(n+1)^s} = \frac{1-2^s}{1+2^s} \sum_{n\geq1} \frac{\varepsilon_n}{n^s},
$$ 
which was shown by Allouche and Cohen \cite{Allouche85} as well as Alkauskas \cite{Alkauskas01} using different techniques. More recently, the present author \cite{Toth22} proved that
\begin{equation}\label{eq:toth}
\sum_{n\geq1} \frac{(2^s+1) t_{n-1} + (2^s-1) t_n}{n^s}  = 2^s \zeta(s),
\end{equation}
valid for all $\Re(s)>1$. On the other hand, Dirichlet series involving the paperfolding sequence also exist, and one particularly interesting example is
\begin{equation}\label{eq:allouchesondow}
\sum_{n\geq0}\frac{\beta_n}{(n+1)^{2k+1}} = \frac{\pi^{2k+1}|E_{2k}|}{(2^{2k+2}-2)(2k)!},
\end{equation}
where
$$
\beta_n = (-1)^{b_n},
$$
which is due to Allouche and Sondow \cite{AlloucheSondow15}. The reader may notice a resemblance to our identity in Theorem \ref{theo-zeta-pi}. In this paper, we provide an alternative proof of identity~\eqref{eq:allouchesondow} using a special case of the Hurwitz zeta function, which we also employ within our other proofs.

\subsection{Structure of this paper}
Our main result, Theorem \ref{theo-zeta-pi}, is the combination of several smaller results. First, we find a closed form expression for the Dirichlet series associated with the sequence $b_n$ using its recurrence relation, in terms of odd powers of $\pi$. We do so by using the Hurwitz zeta function and the polygamma function. Then, we employ an existing result on Dirichlet series associated with the sequence $t_n$. Combined, these results form our proof of Theorem \ref{theo-zeta-pi}.

\section{The Hurwitz zeta function with argument $\frac 34$}
Recall the Hurwitz zeta function,
$$
\zeta (s,a)=\sum _{n=0}^{\infty }{\frac {1}{(n+a)^{s}}},
$$
defined here for rational $a>0$ and $Re(s) > 1$. A central element in our proofs in this paper is a closed form for $\zeta(s, 3/4)$ at positive odd integers $s=2k+1$, and argument $\frac 34$. Some special values of $\zeta(2k+1,p/q)$ for integer $p$ and $q$ (including $\frac pq = \frac34$) were already studied by Adamchik \cite{Adamchik07} using generating functions of trigonometric functions and related results involving the $\psi$ function were given by K\"olbig \cite{Kolbig96.2} using functional properties of the polylogarithm. On the other hand, here we provide another, elementary proof of the following identity that is easily adaptable to other values of $p$ and $q$ for which the polygamma function admits a closed form in terms of known constants and functions.
\begin{lemma} \label{lemma-hurwitz}
For $k \in \mathbb{N}$, we have
$$
\zeta(2k+1,3/4) = - \frac{2^{2k-1}}{(2k)!}  \pi^{2k+1}|E_{2k}| + 2^{2k} (2^{2k+1}-1)\zeta(2k+1) ,
$$
where $E_k$ is the $k^{\rm th}$ Euler number.
\end{lemma}
\begin{proof}
Recall the classical relationship between the Hurwitz zeta function and the polygamma function $\psi^{(k)}(z)$ :
$$
\psi^{(k)}(z) = (-1)^{k+1} k! \zeta(k+1,z).
$$
Thus,
$$
\zeta(2k+1,3/4) = - \frac{1}{(2k)!} \psi^{(2k)}(3/4).
$$
The $\psi$-term has been known since K\"olbig \cite{Kolbig96}:
$$
\psi^{(2k)}(3/4) = 2^{2k-1}\left( \pi^{2k+1}|E_{2k}|-2(2k)!(2^{2k+1}-1)\zeta(2k+1) \right),
$$
where $E_k$ denotes the $k^{\rm th}$ Euler number and $\zeta(k)$ is the Riemann zeta function. Thus, we have
\begin{align*}
\zeta(2k+1,3/4) = - \frac{2^{2k-1}}{(2k)!} \left( \pi^{2k+1}|E_{2k}|-2(2k)!(2^{2k+1}-1)\zeta(2k+1) \right) ,
\end{align*}
which after simplification yields the desired identity.
\end{proof}
A few examples are shown below.
\begin{example}
We have
\begin{align*}
(i) \ \ \ \ \zeta(3,3/4) &= 28 \zeta(3) - \pi^3 , \\
(ii) \ \ \ \ \zeta(5,3/4) &= 496 \zeta(5) - \frac{5}{3} \pi^5.
\end{align*}
\end{example}

\section{Dirichlet series associated with $b_n$}
Now let
$$
\delta(s) = \sum_{n\geq1} \frac{b_n}{n^s},
$$
and note that $\delta(s)$ converges for $\Re(s) > 1$ since the sequence $(b_n)_{n\geq0}$ takes only finitely many values. In fact, it is connected to the Hurwitz zeta function with argument $\frac34$ as follows:
\begin{equation} \label{eq:delta}
(1-2^{-s}) \delta(s) = 4^{-s} \zeta(s, 3/4),
\end{equation}
which is valid for $\Re(s)>1$. We note that equation~\eqref{eq:delta} above can easily be proved by splitting the series to even ($2n$) and odd ($4n+1$, $4n+3$) indexes (see, e.g., Allouche and Shallit \cite[Ex.~27, p.~205]{AlloucheShallit03}). Now simply combining Lemma \ref{lemma-hurwitz} with equation~\eqref{eq:delta} gives us the following result.
\begin{lemma} \label{theo-main}
For all integers $k\geq1$, we have
$$
(2^{4k+1} - 2^{2k}) \sum_{n\geq1}\frac{2 b_n -1}{n^{2k+1}} = - \frac{2^{2k-1}}{(2k)!}  \pi^{2k+1}|E_{2k}|,
$$
where $E_n$ denotes the $n^{\rm th}$ Euler number.
\end{lemma}
\begin{proof}
Let $s=2k+1$. Taking equation~\eqref{eq:delta} and expanding the Hurwitz zeta function term on the right-hand side using Lemma \ref{lemma-hurwitz} gives
$$
4^{2k+1} (1-2^{-(2k+1)}) \delta(2k+1) = - \frac{2^{2k-1}}{(2k)!}  \pi^{2k+1}|E_{2k}| + 2^{2k} (2^{2k+1}-1)\zeta(2k+1) .
$$
Thus,
$$
4^{2k+1} (1-2^{-(2k+1)}) \delta(2k+1) - 2^{2k} (2^{2k+1}-1)\zeta(2k+1) = - \frac{2^{2k-1}}{(2k)!}  \pi^{2k+1}|E_{2k}|.
$$
Now letting
$$
R(n;k) = 4^{2k+1} (1-2^{-(2k+1)}) b_n - 2^{2k} (2^{2k+1}-1)
$$
and expanding the left-hand side gives us
$$
\sum_{n\geq1}\frac{R(n;k)}{n^{2k+1}} = - \frac{2^{2k-1}}{(2k)!}  \pi^{2k+1}|E_{2k}|.
$$
Finally, we can rearrange the terms is $R(n;k)$ in order to form
$$
R(n;k) = (2^{4k+1} - 2^{2k}) (2b_n-1),
$$
which concludes the proof.
\end{proof}
Note that the coefficients $(2^{4k+1} - 2^{2k})$ are available in the OEIS in sequence \seqnum{A079598}. The sequence begins
$$
8, 496, 8128, 130816, 2096128, 33550336, \ldots .
$$
A few examples of Lemma \ref{theo-main} with different values of $k$ are shown below.
\begin{example}
The following equalities hold.
\begin{align*}
(i) \ \ \ \ & 28 \sum_{n\geq1} \frac{2 b_n - 1}{n^3} = - \pi^3, \\
(ii) \ \ \ \ & 496 \sum_{n\geq1} \frac{2 b_n - 1}{n^5} = - \frac53 \pi^5, \\ 
(iii) \ \ \ \ & 8128 \sum_{n\geq1} \frac{2 b_n - 1}{n^7} = - \frac{122}{45} \pi^7.
\end{align*}
\end{example}
We conclude this section by providing an alternative proof for Allouche and Sondow's result, as mentioned in the introduction (equation \eqref{eq:allouchesondow}).
\begin{corollary}
Let $\beta_n = (-1)^{b_n}$ for all $n\geq0$. Then for all $k\geq1$, we have
$$
\sum_{n\geq0}\frac{\beta_n}{(n+1)^{2k+1}} = \frac{\pi^{2k+1}|E_{2k}|}{(2^{2k+2}-2)(2k)!},
$$
where $E_n$ denotes the $n^{\rm th}$ Euler number.
\end{corollary}
\begin{proof}
Notice that for any sequence $(s_n)_{n\geq0}$ with values in $\{0,1\}$, we have $(-1)^{s_n} = 1-2s_n$. Applying this to the sequence $b_n$ and using Lemma \ref{theo-main}, we have
$$
(2^{4k+1} - 2^{2k}) \sum_{n\geq1}\frac{\beta_n}{n^{2k+1}} = \frac{2^{2k-1}}{(2k)!}  \pi^{2k+1}|E_{2k}|,
$$
and dividing both sides by $(2^{4k+1} - 2^{2k})$ yields the desired identity.
\end{proof}

\section{Proof of Theorem \ref{theo-zeta-pi}}

We now have all the tools necessary to prove the theorem shown in the introduction.
\begin{proof}[Proof of Theorem \ref{theo-zeta-pi}]
Recall the identity~\eqref{eq:toth} that we mentioned in the introduction, that for all $\Re(s)>1$, we have
$$
\sum_{n\geq1} \frac{(2^s+1) t_{n-1} + (2^s-1) t_n}{n^s}  = 2^s \zeta(s),
$$
where $t_n$ denotes the Thue-Morse sequence. Taking this identity for odd positive integers $s=2k+1$ and isolating the zeta term gives
$$
2^{-({2k+1})} \sum_{n\geq1} \frac{(2^{2k+1}+1) t_{n-1} + (2^{2k+1}-1) t_n}{n^{2k+1}}  = \zeta(2k+1).
$$
Adding this to our equation in Lemma \ref{theo-main} gives
$$
\sum_{n\geq1} \frac{N(n;k)}{n^{2k+1}}  = \zeta(2k+1) - \frac{2^{2k-1}}{(2k)!}  \pi^{2k+1}|E_{2k}|,
$$
where
$$
N(n;k) = (2^{4k+1} - 2^{2k}) (2 b_n -1) + (2^{-({2k+1})}) ((2^{2k+1}+1) t_{n-1} + (2^{2k+1}-1) t_n).
$$
\end{proof}

\section{Conclusion and further work}

In this short paper we have constructed Dirichlet series whose coefficients are linear combinations of automatic sequences, which allowed us to find identities related to the Riemann zeta function at odd positive integers. Our results are easily adapted to other sequences that satisfy different recurrence relations, making it possible to find such closed forms for known constants and functions. For instance, a recurrence relation resulting in a closed form involving the Hurwitz zeta function at \textit{even} positive integers (as opposed to the results in this paper), would allow evaluating Catalan's constant $C$ in a similar way, using the well-known identity
$$
\zeta(2, 3/4) = \pi^2 - 8C,
$$
which is given for instance by K\"olbig \cite{Kolbig96}, in terms of the polygamma function.

\section{Acknowledgment}

The author is grateful to an anonymous referee, whose remarks and suggestions have improved the quality of this paper, as well as to Prof. Shallit for locating the reference to a formula in N\"orlund's work \cite[Eq.\ 81*, p.\ 66]{Norlund24}.

\bigskip
\hrule
\bigskip

\noindent 2020 {\it Mathematics Subject Classification}: Primary 11M41, Secondary 11B85, 68R15.

\noindent \emph{Keywords:} Dirichlet series, Thue-Morse sequence, paperfolding sequence, dragon-curve sequence, Riemann zeta function.

\bigskip
\hrule
\bigskip

\noindent (Concerned with sequences
\seqnum{A010060},
\seqnum{A014707},
\seqnum{A014577}, and
\seqnum{A079598}.)

\bigskip
\hrule
\bigskip

\end{document}